\documentclass{amsart}
\title[Dolbeault cohomology of  solvmanifolds]
{Techniques of computations of Dolbeault cohomology of  solvmanifolds}

\author{Hisashi Kasuya}

\usepackage{amssymb}
\usepackage{amsmath}
\usepackage{amscd}
\usepackage{amstext}
\usepackage{amsfonts}
\usepackage[all]{xy}
\usepackage{multicol}

\theoremstyle{plain}
\newtheorem{theorem}{Theorem}[section] 
\theoremstyle{assumption}
\newtheorem{Assumption}[theorem]{Assumption} 
\theoremstyle{remark}
\newtheorem{remark}{Remark}
\theoremstyle{lemma}
\newtheorem{lemma}[theorem]{Lemma}
\theoremstyle{definition}
\newtheorem{definition}[theorem]{Definition}
\theoremstyle{proposition}
\newtheorem{proposition}[theorem]{Proposition}
\theoremstyle{corollary}
\newtheorem{corollary}[theorem]{Corollary}
\theoremstyle{remark}

\address[H.kasuya]{Graduate school of mathematical science university of tokyo japan }
\curraddr{}
\email{khsc@ms.u-tokyo.ac.jp}


\newcommand{\C}{\mathbb{C}}
\newcommand{\R}{\mathbb{R}}

\newcommand{\Q}{\mathbb{Q}}
\newcommand{\Z}{\mathbb{Z}}
\newcommand{\g}{\frak{g}}
\newcommand{\n}{\frak{n}}

\begin{document} 

\maketitle
\begin{abstract}
We consider semi-direct products $\C^{n}\ltimes_{\phi}N$ of Lie groups  with  lattices $\Gamma$ such that $N$ are nilpotent Lie groups with left-invariant complex structures.
We compute the Dolbeault cohomology of direct sums of holomorphic line bundles over $G/\Gamma$ by using the Dolbeaut cohomology of the Lie algebras of the direct product $\C^{n}\times N$.
As a corollary of this computation, we can compute the Dolbeault cohomology $H^{p,q}(G/\Gamma)$ of $G/\Gamma$ by using a finite dimensional cochain complexes.
Computing some examples, we observe that the Dolbeault cohomology varies for choices of lattices $\Gamma$.
\end{abstract}
\section{Introduction}
Let $G$ be  a simply connected solvable Lie group and $\g$ the Lie algebra of $G$.
We assume that $G$ admits a lattice $\Gamma$ and a left-invariant complex structure $J$.
We consider the Dolbeault cohomology $H^{\ast,\ast}_{\bar\partial }(G/\Gamma)$ of the complex solvmanifold $G/\Gamma$.
We also consider the cohomology $H^{\ast,\ast}_{\bar\partial}(\g)$ of the differential bigraded algebra (shortly DBA) $\bigwedge^{\ast,\ast} \g^{\ast}$ of the complex valued left-invariant differential forms with the operator $\bar\partial$.

The purpose of this paper is to prove that one can compute the Dolbeault cohomology  of certain class of solvmanifolds $G/\Gamma$ by using finite dimensional DBAs.
In this paper we consider a Lie group $G$ as in the following assumption.

\begin{Assumption}\label{Ass}
$G$ is the semi-direct product $\C^{n}\ltimes _{\phi}N$ so that:\\
(1) $N$ is a simply connected nilpotent Lie group with a left-invariant complex structure $J$.\\
Let $\frak a$ and $\n$ be the Lie algebras of $\C^{n}$ and $N$ respectively.\\
(2) For any $t\in \C^{n}$, $\phi(t)$ is a holomorphic automorphism of $(N,J)$.\\
(3) $\phi$ induces a semi-simple action on the Lie algebra $\n$ of $N$.\\
(4) $G$ has a lattice $\Gamma$. (Then $\Gamma$ can be written by $\Gamma=\Gamma^{\prime}\ltimes_{\phi}\Gamma^{\prime\prime}$ such that $\Gamma^{\prime}$ and $\Gamma^{\prime\prime}$ are  lattices of $\C^{n}$ and $N$ respectively and for any $t\in \Gamma^{\prime}$ the action $\phi(t)$ preserves $\Gamma^{\prime\prime}$.) \\
(5) The inclusion $\bigwedge^{\ast,\ast}\n^{\ast}\subset A^{\ast,\ast}(N/\Gamma^{\prime\prime})$ induces an isomorphism 
\[H^{\ast,\ast}_{\bar\partial}(\n)\cong H^{\ast,\ast}_{\bar\partial }(N/\Gamma^{\prime\prime}).\]
\end{Assumption}

Let $\alpha: \C^{n}\to \C^{\ast}$ be a character (i.e. a representation on $1$-dimensional vector space $\C_{\alpha}$)  of $\C^{n}$ .
By the projection $\C^{n}\ltimes _{\phi}N\to \C^{n}$, we regard $\alpha$ as a character of $G$ as in Assumption \ref{Ass}.
We consider the holomorphic line bundle $L_{\alpha}=(G\times \C_{\alpha})/\Gamma$ and the  Dolbeault complex  $(A^{\ast,\ast}(G/\Gamma,L_{\alpha} ),\bar \partial)$ with values in the line bundle $L_{\alpha}$.
Let  $ {\mathcal L}$ be the set of isomorphism classes of line bundles over $G/\Gamma$ given by  characters of $\C^{n}$.
We consider the direct sum
\[\bigoplus_{L_{\beta}\in {\mathcal L}} A^{\ast,\ast}(G/\Gamma,L_{\beta})
\]
 of Dolbeault complexes.
Then by the wedge products and the tensor products, the direct sum $\bigoplus_{L_{\beta}\in {\mathcal L}} A^{\ast,\ast}(G/\Gamma,L_{\beta})$ is a DBA.

\begin{theorem}\label{MMMTTT}
There exists a subDBA $A^{\ast,\ast}$ of 
\[\bigoplus_{L_{\beta}\in {\mathcal L}} A^{\ast,\ast}(G/\Gamma,L_{\beta})\]
 such that
we have a DBA isomorphism $\iota : \bigwedge ^{\ast,\ast}({\frak a}\oplus \n)^{\ast}\cong A^{\ast,\ast}$ and
 the inclusion 
\[\Phi:A^{\ast,\ast}\to \bigoplus_{L_{\beta}\in {\mathcal L}} A^{\ast,\ast}(G/\Gamma,L_{\beta})\] induces a cohomology isomorphism. 
\end{theorem}
See Section 3 for the construction of $A^{\ast,\ast}$.
\begin{corollary}\label{GFII}
We have the finite dimensional subDBA $B^{\ast,\ast}=\Phi^{-1}(A^{\ast,\ast}(G/\Gamma))$ of $A^{\ast,\ast}(G/\Gamma)$ such that the inclusion $\Phi:B^{\ast,\ast}\to A^{\ast,\ast}(G/\Gamma)$ induces a cohomology isomorphism
\[H^{\ast,\ast}_{\bar \partial}(B^{\ast,\ast})\cong H^{\ast,\ast}_{\bar \partial}(G/\Gamma).
\]
\end{corollary}
See Corollary \ref{CORR} for the construction of $B^{\ast,\ast}$.
\begin{remark}
Let $N$ be a simply connected nilpotent Lie group with a lattice $\Gamma^{\prime\prime}$ and a left-invariant complex structure $J$.
Like Nomizu's theorem (\cite{Nom}) for the de Rham cohomology of nilmanifolds, it is known that an isomorphism $H^{\ast,\ast}_{\bar\partial }(N/\Gamma^{\prime\prime})\cong H^{\ast,\ast}_{\bar\partial}(\n)$ holds if $(N,J,\Gamma^{\prime\prime})$ meet one of the following  conditions:

(N) The complex manifold $N/\Gamma^{\prime\prime}$ has the structure of an iterated principal holomorphic torus bundle (\cite{CF}).

(Q)  $J$  is a small deformation of a rational complex structure i.e. for the rational structure $\n_{\Q}\subset \n$ of the Lie algebra $\n$ induced by a lattice $\Gamma^{\prime\prime}$ (see \cite[Section 2]{R}) we have $J(\n_{\Q})\subset \n_{\Q}$ (\cite{CFI}).

 (C)   $(N,J)$ is a complex  Lie group (\cite{Sak}).\\
\end{remark}

By using Corollary \ref{GFII},  we actually  compute  the Dolbeault cohomology of some examples in Section 5.
Unlike nilmanifolds, we observe that in many cases the Dolbeault cohomology of solvmanifolds can not be completely computed  by using only  Lie algebras.
Moreover we give examples of non-K\"ahler complex solvmanifolds with the Hodge symmetry.

\begin{remark}
If $N$ has  a nilpotent complex structure (see \cite{CF}), then $(\bigwedge^{\ast,\ast}( {\frak a}\oplus\n)^{\ast}, \bar \partial)$ is the minimal model of the DBA $\bigoplus_{L_{\beta}\in {\mathcal L}} A^{\ast,\ast}(G/\Gamma,L_{\beta})$ (see \cite{NT}).
\end{remark}

\section{Holomorphic line bundles over complex tori}
\begin{lemma}\label{twi}
Let $\Gamma$ be a finitely generated free abelian group and $\alpha:\Gamma \to \C^{\ast}$ a character of $\Gamma$.
If the character $\alpha$ is non-trivial, then we have $H^{\ast}(\Gamma, \C_{\alpha})=0$.
\end{lemma}
\begin{proof}
First we assume $\Gamma\cong \Z$.
Then we have 
\[H^{0}(\Z, \C_{\alpha})=\{m\in \C_{\alpha}\vert \alpha(g)m=m,\, {\rm for\,\,all \,\,}g\in \Z \}=0.
\]
Like the  de Rham cohomology of $S^{1}$, by the Poincar\'e duality we have
\[H^{1}(\Z,\C_{\alpha})\cong H^{0}(\Z,\C_{\alpha^{-1}})^{\ast}=0,
\]
and obviously $H^{p}(\Z,\C_{\alpha})=0$ for $p\ge 2$.
Hence the lemma holds in this case.
In  general case, we consider a decomposition $\Gamma=A\oplus B$ such that $A$ is a rank 1 subgroup and the restriction of $\alpha$ on $A$ is also non-trivial.
Then we have the Hochshild-Serre spectral sequence $E_{r}$ such that
\[E_{2}^{p,q}=H^{p}(\Gamma/A, H^{q}(A,\C_{\alpha}))
\]
and this converges to $H^{p+q}(\Gamma, \C_{\alpha})$.
Since $H^{q}(A,\C_{\alpha})=0$ for any $q$, we have $E_{2}=0$ and hence  the lemma follows.
\end{proof}

We consider a complex vector space $\C^{n}$ with a lattice $\Gamma$.
Let  $\alpha:\C^{n}\to \C^{\ast}$ be a $C^{\infty}$-character of $\C^{n}$. 
We have the holomorphic line bundle $L_{\alpha }=(\C^{n}\times \C_{\alpha})/\Gamma $  over the complex torus $\C^{n}/\Gamma$.
We define the equivalence relation  on the space of $C^{\infty}$-characters of $\C^{n}$ such that 
 $\alpha\sim \beta$  if $\alpha\beta^{-1}$ is holomorphic.
\begin{lemma}\label{uniq}
Let  $\alpha:\C^{n}\to \C^{\ast}$ be a  $C^{\infty}$-character of $\C^{n}$.
There exists a unique unitary character $\beta$ such that $\alpha\sim \beta$.
\end{lemma}
\begin{proof}
For a coordinate $(x_{1}+\sqrt{-1}y_{1},\dots ,x_{n}+\sqrt{-1}y_{n})\in \C^{n}$, a character $\alpha$ is written as 
\[\alpha(x_{1}+\sqrt{-1}y_{1},\dots ,x_{n}+\sqrt{-1}y_{n})=\exp (\sum_{i=1}^{n}(a_{i}x_{i}+b_{i}y_{i}+\sqrt{-1}(c_{i}x_{i}+d_{i}y_{i})))\]
for some $a_{i},b_{i},c_{i},d_{i}\in \R^{n}$.
Let $\alpha^{\prime}$ be the holomorphic character defined by
\[\alpha^{\prime}(x_{1}+\sqrt{-1}y_{1},\dots ,x_{n}+\sqrt{-1}y_{n})=\exp(\sum_{i=1}^{n}(-a_{i}(x_{i}+\sqrt{-1}y_{i})+\sqrt{-1}b_{i}(x_{i}+\sqrt{-1}y_{i})).
\]
Then the character $\beta=\alpha\alpha^{\prime}$ is unitary.
If a unitary character is holomorphic, then it is trivial.
Hence such $\beta$ is unique.
\end{proof}

\begin{lemma}{\rm (\cite{Po})}\label{trii}
Let  $\beta:\C^{n}\to \C^{\ast}$ be a unitary $C^{\infty}$-character of $\C^{n}$.
Then the holomorphic line bundle $L_{\beta}$ is trivial if and only if the restriction of $\beta$ on $\Gamma$ is trivial.
\end{lemma}

\begin{proposition}\label{triv}
Let  $\alpha:\C^{n}\to \C^{\ast}$ be a  $C^{\infty}$-character of $\C^{n}$.
 If $L_{\alpha}$ is a non-trivial holomorphic line bundle, then the Dolbeault cohomology $H^{\ast,\ast}_{\bar\partial}(\C^{n}/\Gamma,L_{\alpha})$ with values in the line bundle $L_{\alpha}$ is trivial. 
\end{proposition}
\begin{proof}
Let $\beta$ be the unitary character such that $\alpha\sim \beta$ as in Lemma \ref{uniq}.
Then we have $L_{\alpha}\cong L_{\beta}$.
Let $D$ be the flat connection  on $L_{\beta}$ induced by $\beta$.
We have the decomposition $D=\partial+\bar\partial$ so that $\bar\partial$ is the Dolbeault operator on $L_{\beta}$.
Since $\beta$ is unitary, we have a Hermitian metric on $L_{\beta}$ such that for a K\"ahler metric on $\C^{n}/\Gamma$ we have the standard identity of the Laplacians of $D$ and $\bar\partial$ (see \cite[Section 7]{Gold}).
Hence we have an isomorphism $H^{\ast}_{\bar\partial}(\C^{n}/\Gamma,L_{\beta}) \cong H^{\ast}_{D}(\C^{n}/\Gamma,L_{\beta})$.
If $L_{\beta}$ is non-trivial, then  we have $H^{\ast}_{D}(\C^{n}/\Gamma,L_{\beta})=H^{\ast}(\Gamma,\C_{\beta})=0$ by Lemma \ref{twi}.
Hence the proposition follows.
\end{proof}

\section{The construction of $A^{\ast,\ast}$}
Let $G$ be a Lie group  as in  Assumption \ref{Ass}.
Consider the decomposition $\n_{\C}=\n^{1,0}\oplus \n^{0,1}$.
By the condition (2), this decomposition is a direct sum of $\C^{n}$-modules.
By the condition (3) we have a basis $Y_{1},\dots ,Y_{m}$ of $\n^{1,0}$ such that the action $\phi$ on $\n^{1,0}$ is represented by
$\phi(t)={\rm diag} (\alpha_{1}(t),\dots, \alpha_{m} (t))$.
Since $Y_{j}$ is a left-invariant vector field on $N$,
the vector field $\alpha_{j}Y_{j}$ on $\C^{n}\ltimes _{\phi} N$ is  left-invariant.
Hence we have a basis $X_{1},\dots,X_{n}, \alpha_{1}Y_{1},\dots ,\alpha_{m}Y_{m}$ of $\g^{1,0}$.
Let $x_{1},\dots,x_{n}, \alpha^{-1}_{1}y_{1},\dots ,\alpha_{m}^{-1}y_{m}$ be the  basis of $\bigwedge^{1,0}\g^{\ast}$ which is dual to $X_{1},\dots,X_{n}, \alpha_{1}Y_{1},\dots ,\alpha_{m}Y_{m}$.
Then we have 
\[\bigwedge ^{p,q}\g^{\ast}=\bigwedge ^{p}\langle x_{1},\dots ,x_{n}, \alpha^{-1}_{1}y_{1},\dots ,\alpha^{-1}_{m}y_{m}\rangle\otimes \bigwedge ^{q}\langle \bar x_{1},\dots ,\bar x_{n}, \bar\alpha^{-1}_{1}\bar y_{1},\dots ,\bar\alpha^{-1}_{m}\bar y_{m}\rangle.
\]

Let  $\alpha:\C^{n}\to \C^{\ast}$ be a character of $\C^{n}$.
Let $(A^{\ast,\ast}(G)\otimes \C_{\alpha})^{\Gamma}$ be the space of $\C_{\alpha}$-valued $\Gamma$-invariant differential forms on $G$.
Then we can identify the Dolbeault complex $A^{\ast,\ast}(G/\Gamma,L_{\alpha})$ with $(A^{\ast,\ast}(G)\otimes \C_{\alpha})^{\Gamma}$.
Hence for $\omega\in \bigwedge ^{\ast,\ast}\g^{\ast}$ and $v_{\alpha}\in\C_{\alpha}$, we have \[\omega\otimes (\alpha^{-1}v_{\alpha} )\in A^{\ast,\ast}(G/\Gamma,L_{\alpha}).\]

Let $ {\mathcal L}$ be the set as in Introduction.
By Section 2, we can regard $ {\mathcal L}$ as the set of   isomorphism classes of line bundles over $G/\Gamma$ given by unitary  characters of $\C^{n}$.
We consider the DBA $\bigoplus_{L_{\alpha}\in {\mathcal L}} A^{\ast,\ast}(G/\Gamma,L_{\alpha})$.
We define the DBA $A^{\ast,\ast}$ to prove Theorem \ref{MMMTTT}.
\begin{definition}\label{DeA}
Let  $x_{1},\dots,x_{n}, \alpha^{-1}_{1}y_{1},\dots ,\alpha_{m}^{-1}y_{m}$ be the  basis of $\bigwedge^{1,0}\g^{\ast}$ as above.
By Lemma \ref{uniq}, we have the unitary character $\beta_{j}$ such that $\alpha_{j}\sim\beta_{j}$.
We consider the holomorphic line bundles $L_{\beta^{-1}_{j}}$ over $G/\Gamma$.
 By 
$A^{\ast,\ast}(G/\Gamma,L_{\beta^{-1}_{j}})=(A^{\ast,\ast}(G)\otimes \C_{\beta_{j}^{-1}})^{\Gamma}$,
for $\C_{\beta_{j}^{-1}}\ni v_{\beta_{j}^{-1}}\not=0$ we consider 
 \[ \alpha^{-1}_{j}y_{j}\otimes (\beta_{j}v_{\beta^{-1}_{j}}) \in A^{\ast,\ast}(G/\Gamma,L_{\beta_{j}^{-1}}).\]
Let $A^{\ast,\ast}$ be the subDBA of $\bigoplus_{L_{\alpha}\in {\mathcal L}} A^{\ast,\ast}(G/\Gamma,L_{\alpha})$ defined by
\begin{multline*}A^{p,q}
=\bigwedge ^{p}\langle x_{1},\dots ,x_{n}, \alpha^{-1}_{1} y_{1}\otimes (\beta_{1} v_{\beta_{1}^{-1}}) ,\dots ,\alpha^{-1}_{m}y_{m}\otimes (\beta_{m}v_{\beta_{m}^{-1}})\rangle\\
\bigotimes \bigwedge ^{q}\langle \bar x_{1},\dots ,\bar x_{n}, \bar\alpha^{-1}_{1}\bar y_{1}\otimes(\gamma_{1}v_{\gamma_{1}^{-1}}),\dots ,\bar\alpha^{-1}_{m}\bar y_{m} \otimes(\gamma_{m}v_{\gamma_{m}^{-1}}) \rangle.
\end{multline*}
\end{definition}

\begin{lemma}\label{ISS}

Let $\iota:\bigwedge^{\ast,\ast} ({\frak a}\oplus \n)^{\ast}\to A^{\ast,\ast}$ be the algebra homomorphism defined by
\[\iota(x_{i})=x_{i},\]
\[\iota( \alpha^{-1}_{j}y_{j})=\alpha^{-1}_{j}y_{j}\otimes \beta_{j}v_{\beta^{-1}_{j}}.\]
Then we have a DBA isomorphism
\[\iota: (\bigwedge^{\ast,\ast} ({\frak a}\oplus \n)^{\ast},\bar \partial)\cong (A^{\ast,\ast},\bar\partial).\]
\end{lemma}
\begin{proof}
Since $\alpha^{-1}_{j}\beta_{j}$ is holomorphic, we have
\[\bar\partial(\alpha^{-1}_{j}y_{j}\otimes \beta_{j}v_{\beta^{-1}_{j}})=\alpha^{-1}_{j}(\bar\partial y_{j})\otimes \beta_{j}v_{\beta^{-1}_{j}}.
\]
This implies $\bar\partial \circ \iota=\iota \circ  \bar\partial$.
Hence the lemma follows.
\end{proof}

Let $g$ be the left-invariant  Hermitian metric on $G$ defined by 
\[g=x_{1}\bar x_{1}+\dots +x_{n}\bar x_{n}+\alpha_{1}^{-1}\bar\alpha_{1}^{-1}y_{1}\bar y_{1}+\dots ++\alpha_{m}^{-1}\bar\alpha_{m}^{-1}y_{m}\bar y_{m}.\]
Let  $\beta:\C^{n}\to \C^{\ast}$ be a unitary $C^{\infty}$-character of $\C^{n}$.
Take $\C_{\beta}\ni v_{\beta}\not=0$.
Then $\beta^{-1}v_{\beta}$ is a $C^{\infty}$-frame of the line bundle $L_{\beta}=(G\times\C_{\beta})/\Gamma$.
We define the Hermitian metric $h_{\beta}$ on $L_{\beta}$ such that $h_{\beta}(\beta^{-1}v_{\beta},\beta^{-1}v_{\beta})=1$.
Let $\bar\ast_{g\otimes h_{\beta}}:A^{p,q}(G/\Gamma,L_{\beta})\to A^{n+m-p,n+m-q}(G/\Gamma,L^{\ast}_{\beta})$ be the $\C$-anti-linear Hodge star operator of $g\otimes h_{\beta}$ on $A^{\ast,\ast}(G/\Gamma,L_{\beta})$ and let 
\[\bar\delta=\bar\ast_{g\otimes h_{\beta}}\circ \bar\partial \circ \bar \ast_{g\otimes h_{\beta}}, \,\,\,\,\,
 \Box_{g\otimes h_{\beta}}=\bar\partial \bar\delta+\bar\delta\bar\partial\]
 and 
\[{\mathcal H}^{p,q}(G/\Gamma, L_{\beta})=\{\omega\in A^{\ast,\ast}(G/\Gamma,L_{\beta})\vert \Box_{g\otimes h_{\beta}}\omega=0\}.\] 
We consider the $\bar\partial $-Laplace operator $\oplus\Box_{g\otimes h_{\beta}}$ on the direct sum $\bigoplus_{L_{\beta}\in {\mathcal L}} A^{\ast,\ast}(G/\Gamma,L_{\beta})$.

We consider the basis $x_{1},\dots,x_{n},y_{1},\dots, y_{m}$ of $\bigwedge^{1,0}({\frak a}\oplus\n)^{\ast}$.
Let $g^{\prime}$ be the Hermitian metric on ${\frak a}\oplus\n$ defined by
 \[g^{\prime}=x_{1}\bar x_{1}+\dots +x_{n}\bar x_{n}+y_{1}\bar y_{1}+\dots +y_{m}\bar y_{m}.\]
Let $\bar\ast_{g^{\prime}}:\bigwedge ^{p,q}({\frak a}\oplus \n)^{\ast}\to \bigwedge ^{n+m-p,n+m-q}({\frak a}\oplus \n)^{\ast}$ be the $\C$-anti-linear  Hodge star operator of $g^{\prime}$ on $\bigwedge ^{\ast,\ast}({\frak a}\oplus \n)^{\ast}$ and let 
\[\bar\delta=\bar\ast_{g^{\prime}}\circ \bar\partial \circ \bar \ast_{g^{\prime}},\,\,\,\,\, 
 \Box_{g^{\prime}}=\bar\partial \bar\delta+\bar\delta\bar\partial \]
 and 
\[{\mathcal H}^{p,q}({\frak a}\oplus \n)=\{\omega\in\bigwedge ^{\ast,\ast}({\frak a}\oplus \n)^{\ast}\vert \Box_{g^{\prime}}\omega=0\}.\]

\begin{lemma}\label{lap}
We consider the isomorphism $\iota: \bigwedge ^{\ast,\ast}({\frak a}\oplus \n)^{\ast} \cong A^{\ast,\ast}$ as in  Lemma \ref{ISS}.
Then we have 
\[\iota \circ \Box_{g^{\prime}}=(\oplus\Box_{g\otimes h_{\alpha}})\circ \iota.\]
\end{lemma}
\begin{proof}
Let $\oplus \bar\ast_{g\otimes h_{\alpha}}$ be  the Hodge star operator on $\bigoplus_{L_{\beta}\in {\mathcal L}} A^{\ast,\ast}(G/\Gamma,L_{\beta})$.
It is sufficient to show 
\[\iota \circ \bar \ast_{g^{\prime}}=(\oplus \bar\ast_{g\otimes h_{\alpha}})\circ \iota.\]
For a multi-index $I=\{i_{1},\dots ,i_{r}\}$, we write $x_{I}=x_{i_{1}}\wedge \dots \wedge x_{i_{r}}$, $y_{I}=y_{i_{1}}\wedge \dots \wedge y_{i_{r}}$, $\alpha_{I}=\alpha_{i_{1}} \cdots  \alpha_{i_{r}}$ and $\beta_{I}=\beta_{i_{1}} \cdots  \beta_{i_{r}}$.
For multi-indices $I,K \subset \{1,\dots,n\}$ and $J,L\subset \{1,\dots,m\}$,  we have 
\[\bar\ast_{g^{\prime}}(x_{I}\wedge y_{J}\wedge \bar x_{K}\wedge \bar y_{L})=\epsilon x_{I^{\prime}}\wedge y_{J^{\prime}}\wedge \bar x_{K^{\prime}}\wedge \bar y_{L^{\prime}}
\]
where $I^{\prime}$, $J^{\prime}$, $K^{\prime}$ and $L^{\prime}$ are complements  and $\epsilon$ is the sign of a permutation.
We also have 
\begin{multline*}
\oplus \bar\ast_{g\otimes h_{\alpha}}(x_{I}\wedge \alpha^{-1}_{J}y_{J}\wedge \bar x_{K}\wedge \bar \alpha^{-1}_{L}\bar y_{L}\otimes \beta_{J}\gamma_{L}v_{\beta^{-1}_{J}\gamma^{-1}_{L}})\\
=\epsilon x_{I^{\prime}}\wedge \alpha^{-1}_{J^{\prime}}y_{J^{\prime}}\wedge \bar x_{K^{\prime}}\wedge \bar \alpha^{-1}_{L^{\prime}}\bar y_{L^{\prime}}\otimes \beta^{-1}_{J}\gamma^{-1}_{L}v_{\beta_{J}\gamma_{L}}.
\end{multline*}
Hence we only need to show 
\[\beta^{-1}_{J}\gamma^{-1}_{L}=\beta_{J^{\prime}}\gamma_{L^{\prime}}.
\]
Since a Lie group with a lattice  is  unimodular (see \cite[Remark 1.9]{R}), the action $\phi$ on $\n$ is represented by  unimodular matrices.
Hence we have $\alpha_{J}\bar\alpha_{L}\alpha_{J^{\prime}}\bar\alpha_{L^{\prime}}=1$.
 This implies $ \beta^{-1}_{J}\gamma^{-1}_{L}=\beta_{J^{\prime}}\gamma_{L^{\prime}}$.
Hence the lemma follows.
\end{proof}
\begin{corollary}\label{Ijn}
The inclusion 
\[\Phi:A^{\ast,\ast}\to \bigoplus_{L_{\beta}\in {\mathcal L}} A^{\ast,\ast}(G/\Gamma,L_{\beta})\] induces an  injection 
\[H^{p,q}_{\bar\partial}( {\frak a}\oplus\n)\cong H^{p,q}( A^{\ast,\ast})\to H^{p,q}_{\bar \partial}(\bigoplus_{L_{\beta}\in {\mathcal L}} A^{\ast,\ast}(G/\Gamma,L_{\beta})).\]
\end{corollary}
\begin{proof}
We have  isomorphisms ${\mathcal H}^{p,q}(G/\Gamma, L_{\beta})\cong H^{p,q}_{\bar\partial}(G/\Gamma,L_{\beta})$ and ${\mathcal H}^{p,q}({\frak a}\oplus \n)\cong H^{p,q}_{\bar\partial}({\frak a}\oplus \n)$ (see \cite{RO}).
By Lemma \ref{lap}, we have 
\[\iota({\mathcal H}^{p,q}({\frak a}\oplus \n))\subset \bigoplus_{L_{\beta}\in {\mathcal L}}{\mathcal H}^{p,q}(G/\Gamma, L_{\beta}).\]
Hence the corollary follows.
\end{proof}

\section{Proof of the main theorem}\label{MTT}

\begin{proposition}\label{FIIB}
Let $G$ be a Lie group as in Assumption \ref{Ass}.
$G/\Gamma$ is a holomorphic fiber bundle over a torus  with a nilmanifold as a  fiber,
\[N/\Gamma^{\prime\prime}\to G/\Gamma\to \C^{n}/\Gamma^{\prime}\]
  such that the structure group of this fibration is discrete. 
\end{proposition}
\begin{proof}
Consider the covering $\C^{n}\times (N/\Gamma^{\prime\prime})\to G/\Gamma$ such that the covering transformation is the action of $\Gamma^{\prime}$ on $\C^{n}\times (N/\Gamma^{\prime\prime})$ given by $g\cdot (a,b)=(a+g,\phi(g)b)$.
Hence we have the fiber bundle $G/\Gamma\to \C^{n}/\Gamma^{\prime}$ with the fiber $N/\Gamma^{\prime\prime}$ and the discrete structure group $\phi(\Gamma^{\prime})\subset {\rm Aut}(N)$.
Since $\phi(g)$ is a holomorphic automorphism, this fiber bundle is holomorphic. 
\end{proof}

\begin{proof}[\bf Proof of Theorem \ref{MMMTTT}]
For $L_{\beta}\in {\mathcal L}$, by Borel's results in \cite[Appendix 2]{Hir}, we have the spectral sequence $(E_{r},d_{r})$ of the filtration of $A^{p,q}(G/\Gamma,L_{\beta})$ induced by the  holomorphic fiber bundle $p:G/\Gamma\to \C^{n}/\Gamma^{\prime}$ as in Proposition \ref{FIIB} such that:\\
(1)$E_{r}$ is $4$-graded, by the fiber-degree, the base-degree and the type. Let $\ ^{p,q}E_{r}^{s,t}$ be 
the subspace of	elements of $E_{r}$ of type $(p,q)$, fiber-degree $s$ and	base-degree $t$.	We have $\ ^{p,q}E_{r}^{s,t}=0$ if $p + q = s + t$ or if one of $p, q, s, t$ is negative.\\
(2) If $p+q=s+t$, then we have 
\[\ ^{p,q}E^{s,t}_{2}\cong \sum_{i\ge 0}H^{i,i-s}_{\bar\partial}(\C^{n}/\Gamma^{\prime}, L_{\beta}\otimes{\bf H}^{p-i, q-s+i}(N/\Gamma^{\prime\prime}))
\]
where ${\bf H}^{p-i, q-s+i}(N/\Gamma^{\prime\prime})$ is the holomorphic fiber bundle $\bigcup_{b\in \C^{n}/\Gamma^{\prime}}H^{p,q}_{\bar\partial}(p^{-1}(b))$.\\
(3)The spectral sequence converges to $H_{\bar\partial }(G/\Gamma, L_{\beta})$.

By the assumption $H^{\ast,\ast}_{\bar\partial}(\n)\cong H^{\ast,\ast}_{\bar\partial }(N/\Gamma^{\prime\prime})$, the fiber bundle ${\bf H}^{p-i, q-s+i}(N/\Gamma^{\prime\prime})$ is the holomorphic vector bundle with the fiber $H^{p-i, q-s+i}_{\bar\partial}(\n)$ induced by the action $\phi$ of $\Gamma$ on $H^{p-i, q-s+i}_{\bar\partial}(\n)$.
Since the action $\phi$ on $\n$ is semi-simple, the action  of $\C^{n}$ on $H^{p-i, q-s+i}_{\bar\partial}(\n)$ induced by $\phi$ is diagonalizable.
The fiber bundle splits as ${\bf H}^{p-i, q-s+i}(N/\Gamma^{\prime\prime})=\oplus L_{\delta_{j}}$ for some $ L_{\delta_{j}}\in {\mathcal L}$.
Hence we have 
\[H^{i,i-s}_{\bar\partial}(\C^{n}/\Gamma^{\prime}, L_{\beta}\otimes{\bf H}^{p-i, q-s+i}(N/\Gamma^{\prime\prime}))=H^{i,i-s}_{\bar\partial}(\C^{n}/\Gamma^{\prime}, \bigoplus_{\delta_{j}} L_{\beta}\otimes L_{\delta_{j} }).\]
By Proposition \ref{triv}, we have $H^{i,i-s}_{\bar\partial}(\C^{n}/\Gamma^{\prime},  L_{\beta}\otimes L_{\delta_{j}})\cong H^{i,i-s}_{\bar\partial}(\C^{n}/\Gamma^{\prime})$ if $L_{\beta}\otimes L_{\delta_{j}}$ is trivial and $H^{i,i-s}_{\bar\partial}(\C^{n}/\Gamma^{\prime},  L_{\beta}\otimes L_{\delta_{j}}))=0$ if $L_{\beta}\otimes L_{\delta_{j}}$ is non-trivial.
Hence we have 
\[H^{i,i-s}_{\bar\partial}(\C^{n}/\Gamma^{\prime}, \bigoplus_{L_{\beta}\in {\mathcal L}} L_{\beta}\otimes{\bf H}^{p-i, q-s+i}(N/\Gamma^{\prime\prime}))
\cong H^{i,i-s}_{\bar\partial}(\C^{n}/\Gamma^{\prime})\otimes H^{p-i, q-s+i}_{\bar\partial}(\n).
\]
For the direct sum $\bigoplus_{L_{\beta}\in {\mathcal L}} A^{\ast,\ast}(G/\Gamma,L_{\beta})$, we consider this spectral sequence $E_{r}$.
Then we have 
\begin{multline*}
\ ^{p,q}E^{s,t}_{2}\cong \sum_{i\ge 0}H^{i,i-s}_{\bar\partial}(\C^{n}/\Gamma^{\prime}, \bigoplus_{L_{\beta}\in {\mathcal L}}  L_{\beta}\otimes{\bf H}^{p-i, q-s+i}(N/\Gamma^{\prime\prime}))\\
\cong  
\sum_{i\ge 0}H^{i,i-s}_{\bar\partial}(\C^{n}/\Gamma^{\prime})\otimes H^{p-i, q-s+i}_{\bar\partial}(\n)
\end{multline*}
This implies an isomorphism $E_{2}\cong \bigoplus_{p,q} H^{p,q}_{\bar\partial}( {\frak a}\oplus\n)$.
On the other hand, by Corollary \ref{Ijn}, we have an injection
\[  H^{p,q}_{\bar\partial}({ \frak a}\oplus\n )\to  H^{p,q}_{\bar \partial}(\bigoplus_{L_{\beta}\in {\mathcal L}} A^{\ast,\ast}(G/\Gamma,L_{\beta})\cong E_{\infty} .
\]
Hence the spectral sequence degenerates at $E_{2}$ and the theorem follows.
\end{proof}

\begin{corollary}\label{CORR}
Let  $B^{\ast,\ast}\subset A^{\ast,\ast}(G/\Gamma)$ be the subDBA of $A^{\ast,\ast}(G/\Gamma)$ given by
\[B^{p,q}=\left\langle x_{I}\wedge \alpha^{-1}_{J}\beta_{J}y_{J}\wedge \bar x_{K}\wedge \bar \alpha^{-1}_{L}\gamma_{L}\bar y_{L}{\Big \vert} \begin{array}{cc}\vert I\vert+\vert K\vert=p,\, \vert J\vert+\vert L\vert=q \\  {\rm the \, \,  restriction \, \,  of }\, \, \beta_{J}\gamma_{L}\, \,  {\rm on  \, \, \Gamma \, \, is\, \,  trivial}\end{array}\right\rangle.
\]
Then  the inclusion $B^{\ast,\ast}\subset A^{\ast,\ast}(G/\Gamma)$ induces a cohomology isomorphism
\[H^{\ast,\ast}_{\bar \partial}(B^{\ast,\ast})\cong H^{\ast,\ast}_{\bar \partial}(G/\Gamma).
\]
\end{corollary}
\begin{proof}
By Lemma \ref{trii}, 
\[\Phi(x_{I}\wedge \alpha^{-1}_{J}y_{J}\wedge \bar x_{K}\wedge \bar \alpha^{-1}_{L}\bar y_{L}\otimes \beta_{J}\gamma_{L}v_{\beta^{-1}_{J}\gamma^{-1}_{L}})\in A^{\ast,\ast}(G/\Gamma)\] if and only if the restriction of $\beta_{J}\gamma_{L}$ on $\Gamma$ is trivial.
Hence we have  $\Phi^{-1}(A^{\ast,\ast}(G/\Gamma))=B^{\ast,\ast}$.
\end{proof}

\begin{remark}
Suppose $\phi:\C^{n}\to {\rm Aut}(\n^{1,0})$ is a holomorphic map.
Since each $\alpha_{j}$ is holomorphic,
$\beta_{j}$ is trivial.
Hence we have $B^{p,0}=\bigwedge^{p,0}\g^{\ast}$.
Moreover if $N$ is a complex Lie group, then $G=\C^{n}\ltimes_{\phi}N$ is also a complex Lie group and any element of $B^{1,0}=\g^{1,0}$ is holomorphic and hence $\bar\partial B^{p,0}=0$.
Hence we have an isomorphism
\[H^{p,q}(G/\Gamma)\cong \bigwedge^{p}\g^{1,0}\otimes H^{q}_{\bar \partial}( B^{0,q}).\]
\end{remark}

\begin{remark}
We suppose the following condition:\\
($\star$) For multi-indices $J$, $L$, if the restriction of $\beta_{J}\gamma_{L}$ on $\Gamma$ is trivial, then  $\beta_{J}\gamma_{L}$ itself is trivial.\\
Then we have $B^{\ast,\ast}\subset \bigwedge ^{\ast,\ast}\g^{\ast}$ and hence  we have an isomorphism 
\[H^{\ast,\ast}_{\bar\partial}(\g)\cong H^{\ast,\ast}_{\bar \partial}(G/\Gamma).
\]
\end{remark}

\section{Examples}
\subsection{Example 1}
Let $G=\C\ltimes _{\phi}\C^{2}$ such that $\phi(x+\sqrt{-1}y)=\left(
\begin{array}{cc}
e^{x}& 0  \\
0&    e^{-x}  
\end{array}
\right)$.
Then for some $a\in \R$  the matrix $\left(
\begin{array}{cc}
e^{a}& 0  \\
0&    e^{-a}  
\end{array}
\right)$
 is conjugate to an element of $SL(2,\Z)$.
 Hence for any $0\not=b\in \R$ we have a lattice $\Gamma=(a\Z+b\sqrt{-1}\Z )\ltimes \Gamma^{\prime\prime}$ such that $\Gamma^{\prime\prime} $ is a lattice of $\C^{2}$.
 Then for a coordinate $(z_{1}=x+\sqrt{-1}y,z_{2},z_{3})\in \C\ltimes _{\phi}\C^{2}$  we have
\[\bigwedge ^{p,q}\g^{\ast}
=\bigwedge ^{p,q}\langle dz_{1},\, e^{-x}dz_{2}, e^{x}dz_{3}\rangle\otimes \langle dz_{1},\, e^{-x}d\bar z_{2},\, e^{x}d\bar z_{3}\rangle.
\]
Since we have  $e^{x}\sim e^{-\sqrt{-1}y}$, the subDBA 
\[A^{\ast,\ast}\subset \bigoplus_{L_{\beta}\in {\mathcal L}} A^{\ast,\ast}(G/\Gamma,L_{\beta})\]
as in Definition \ref{DeA}  is given by 
\begin{multline*}A ^{p,q}
=\bigwedge ^{p,q}\langle dz_{1},\, e^{-x}dz_{2}\otimes e^{-\sqrt{-1}y}v_{e^{\sqrt{-1}y}},\, e^{x}dz_{3} \otimes e^{\sqrt{-1}y}v_{e^{-\sqrt{-1}y}}\rangle \\
\bigotimes \langle d\bar z_{1},\, e^{-x}d\bar z_{2}\otimes e^{-\sqrt{-1}y}v_{e^{\sqrt{-1}y}},\, e^{x}d\bar z_{3}\otimes e^{\sqrt{-1}y}v_{e^{-\sqrt{-1}y}}\rangle.
\end{multline*}
$B^{p,q}\subset A^{p,q}(G/\Gamma)$ varies for a choice of $b\in \R$ as the following.\\
(A) If $b=2n\pi$ for $n\in \Z$, then we have:
\[B^{p,q}= \bigwedge ^{p,q}\langle dz_{1},\, e^{-x-\sqrt{-1}y}dz_{2}, e^{x+\sqrt{-1}y}dz_{3}\rangle\otimes \langle d\bar z_{1},\, e^{-x-\sqrt{-1}y}d\bar z_{2}, e^{x+\sqrt{-1}y}d\bar z_{3}\rangle.\]
(B) If $b=(2n-1)\pi$ for $n\in \Z$, then  we have:
\[B^{1,0}=\langle dz_{1}\rangle,\,\, B^{0,1}= \langle d\bar z_{1}\rangle,
\]
\[B^{2,0}= \langle dz_{2}\wedge dz_{3} \rangle ,\, B^{0,2}= \langle d\bar z_{2}\wedge d\bar z_{3} \rangle ,
\]
 \[B^{1,1}=  \langle  dz_{1}\wedge  d\bar z_{1}, \, e^{-2x-2\sqrt{-1}y}dz_{2}\wedge d\bar z_{2},\, e^{2x+2\sqrt{-1}y}dz_{3}\wedge d\bar z_{3},\, dz_{2}\wedge d\bar z_{3},\, dz_{3}\wedge d\bar z_{2} \rangle,
 \]
\[B^{3,0}= \langle  dz_{1}\wedge dz_{2}\wedge dz_{3}\rangle, \]
\begin{multline*}B^{2,1}=\langle dz_{2}\wedge dz_{3} \wedge d\bar z_{1} ,\, e^{-2x-2\sqrt{-1}y}dz_{1}\wedge dz_{2}\wedge d\bar z_{2},\\
e^{2x+2\sqrt{-1}y}dz_{1}\wedge dz_{3}\wedge d\bar z_{3} ,\, dz_{1}\wedge dz_{2}\wedge d\bar z_{3}, dz_{1}\wedge dz_{3}\wedge d\bar z_{2} \rangle,
\end{multline*}
\begin{multline*} B^{1,2} = \langle dz_{1}\wedge  d\bar z_{2}\wedge d\bar z_{3},\, e^{-2x-2\sqrt{-1}y}dz_{2}\wedge d\bar z_{1} \wedge d\bar z_{2}, \\ e^{2x+2\sqrt{-1}y}dz_{3}\wedge d\bar z_{1}\wedge d\bar z_{3}, \,
dz_{2}\wedge d\bar z_{3}\wedge d\bar z_{1},\, dz_{3}\wedge d\bar z_{2}\wedge d\bar z_{1} \rangle,
\end{multline*}
\[
B^{0,3}= \langle d \bar z_{1}\wedge d\bar z_{2}\wedge d\bar z_{3} \rangle ,
\]
\[B^{3,1}= \langle  dz_{1}\wedge dz_{2}\wedge dz_{3}\wedge d\bar z_{1}\rangle,\, B^{1,3}=\langle d \bar z_{1}\wedge d\bar z_{2}\wedge d\bar z_{3} \wedge d\bar z_{1}\rangle,
\]
\begin{multline*}B^{2,2}=\langle dz_{1}\wedge  dz_{2}\wedge d\bar z_{1} \wedge d\bar z_{3},\\
e^{-2x-2\sqrt{-1}y}dz_{1}\wedge dz_{2}\wedge d\bar z_{1} \wedge d\bar z_{2},\, e^{2x+2\sqrt{-1}y}dz_{1}\wedge dz_{3}\wedge d\bar z_{1}\wedge d\bar z_{3}, \\
dz_{2}\wedge  dz_{3}\wedge d\bar z_{2}\wedge d\bar z_{3}, dz_{1}\wedge dz_{3}\wedge d\bar z_{1}\wedge d\bar z_{2}\rangle,
\end{multline*}
\[B^{3,2}=\langle  dz_{2}\wedge dz_{3}\wedge d\bar z_{1} \wedge d\bar z_{2} \wedge d\bar z_{3}\rangle, B^{2,3}=\langle dz_{1}\wedge dz_{2}\wedge dz_{3} \wedge d\bar z_{2}\wedge d\bar z_{3}\rangle,
\]
\[B^{3,3}=\langle dz_{1}\wedge dz_{2}\wedge dz_{3} \wedge d\bar z_{1}\wedge d\bar z_{2} \wedge d\bar z_{3}\rangle.
\]
 (C) If $b\not=n\pi$ for any $n\in\Z$, then we have:
 \[B^{1,0}=\langle dz_{1}\rangle,\,\, B^{0,1}= \langle d\bar z_{1}\rangle,
\]
\[B^{2,0}= \langle dz_{2}\wedge dz_{3} \rangle ,\, B^{0,2}= \langle d\bar z_{2}\wedge d\bar z_{3} \rangle ,
\]
 \[B^{1,1}=  \langle  dz_{1}\wedge  d\bar z_{1},  dz_{2}\wedge d\bar z_{3}, dz_{3}\wedge d\bar z_{2} \rangle,
 \]
\[B^{3,0}= \langle  dz_{1}\wedge dz_{2}\wedge dz_{3}\rangle, B^{2,1}=\langle dz_{2}\wedge dz_{3} \wedge d\bar z_{1}, \,  dz_{1}\wedge dz_{2}\wedge d\bar z_{3}, \, dz_{1}\wedge dz_{3}\wedge d\bar z_{2} \rangle,
\]
\[B^{1,2} = \langle dz_{1}\wedge  d\bar z_{2}\wedge d\bar z_{3},\, dz_{2}\wedge d\bar z_{3}\wedge d\bar z_{1}, dz_{3}\wedge d\bar z_{2}\wedge d\bar z_{1} \rangle, B^{0,3}= \langle d \bar z_{1}\wedge d\bar z_{2}\wedge d\bar z_{3} \rangle ,
\]
\[B^{3,1}= \langle  dz_{1}\wedge dz_{2}\wedge dz_{3}\wedge d\bar z_{1}\rangle, B^{1,3}=\langle d \bar z_{1}\wedge d\bar z_{2}\wedge d\bar z_{3} \wedge d\bar z_{1}\rangle,
\]
 \[B^{2,2}=\langle dz_{1}\wedge  dz_{2}\wedge d\bar z_{1} \wedge d\bar z_{3}, dz_{2}\wedge dz_{3}\wedge d\bar z_{2}\wedge d\bar z_{3}, dz_{1}\wedge dz_{3}\wedge d\bar z_{1}\wedge d\bar z_{2}\rangle,
\]
\[B^{3,2}=\langle  dz_{2}\wedge dz_{3}\wedge d\bar z_{1} \wedge d\bar z_{2} \wedge d\bar z_{3}\rangle, B^{2,3}=\langle dz_{1}\wedge dz_{2}\wedge dz_{3} \wedge d\bar z_{2}\wedge d\bar z_{3}\rangle,
\]
\[B^{3,3}=\langle dz_{1}\wedge dz_{2}\wedge dz_{3} \wedge d\bar z_{1}\wedge d\bar z_{2} \wedge d\bar z_{3}\rangle.
\]
By Corollary \ref{CORR}, for each case we have an isomorphism $ H^{p,q}_{\bar \partial}(G/\Gamma)\cong B^{p,q}$.
Moreover considering the left-invariant  Hermitian metric $g=dz_{1}d\bar z_{1}+e^{-2x}dz_{2}d\bar z_{2}+ e^{2x}dz_{3}d\bar z_{3}$, we have ${\mathcal H}^{p,q}(G/\Gamma)\cong B^{p,q}$.

\begin{remark}
In the case (A),  the Dolbeault cohomology $ H^{\ast,\ast}_{\bar \partial}(G/\Gamma)$ is isomorphic to the Dolbeault cohomology of complex $3$-torus.
But $G/\Gamma$ is not homeomorphic to a complex $3$-torus.
Moreover considering the metric $g$, the space of the harmonic forms does not satisfy Hodge symmetry (i.e.  $\bar {\mathcal H}^{p,q}(G/\Gamma)\not = {\mathcal H}^{q,p}(G/\Gamma)$).
\end{remark}
\begin{remark}
By Hattori's result in \cite{Hatt}, we have an isomorphism  $H^{\ast}(G/\Gamma)\cong H^{\ast}(\g)$ of the de Rham cohomology of $G/\Gamma$ and the Lie algebra cohomology.
Hence considering the space ${\mathcal H}^{k}_{d}(\g)$  of left-invariant $d$-harmonic forms of the left-invariant Hermitian metric $g$, we have $ {\mathcal H}_{d}^{k}(\g)\cong {\mathcal H}^{k}_{d}(G/\Gamma)$.
By simple computations, in the case (C) we have the Hodge decomposition ${\mathcal H}^{k}_{d}(G/\Gamma)= \bigoplus_{p+q=k}{\mathcal H}^{p,q}(G/\Gamma)$.
Hence $G/\Gamma$ has  cohomological  properties (for example the Fr\"olicher spectral sequence degenerates at $E_{1}$)  of compact K\"ahler manifolds.
But by Arapura's result (solving Benson-Gordon's conjecture) in \cite{Ara}, $G/\Gamma$ admits no K\"ahler structure.
\end{remark}
\begin{remark}
In the case (C), an isomorphism $H^{\ast,\ast}_{\bar\partial}(\g)\cong H^{\ast,\ast}_{\bar \partial}(G/\Gamma)$ holds.
But in the other cases, this isomorphism  does not hold.
\end{remark}
\subsection{Example 2}
Let $G=\C\ltimes_{\phi} \C^{2}$ such that \[\phi(x+\sqrt{-1}y)=\left(
\begin{array}{cc}
e^{x+\sqrt{-1}y}& 0  \\
0&    e^{-x-\sqrt{-1}y}  
\end{array}
\right).\]
Then we have $a+\sqrt{-1}b, c+\sqrt{-1}d\in \C$ such that $ \Z(a+\sqrt{-1}b)+\Z(c+\sqrt{-1}d)$ is a lattice in $\C$ and
$\phi(a+\sqrt{-1}b)$ and $\phi(c+\sqrt{-1}d)$
 are conjugate to elements of $SL(4,\Z)$ where we regard  $SL(2,\C)\subset SL(4,\R)$ (see \cite{Hd}).
Hence we have a lattice $\Gamma=(\Z(a+\sqrt{-1}b)+\Z( c+\sqrt{-1}d))\ltimes_{\phi} \Gamma^{\prime\prime}$ such that $\Gamma^{\prime\prime}$ is a lattice of $\C^{2}$.
For a coordinate $(z_{1},z_{2},z_{3})\in \C\ltimes \C^{2}$, we have
\[\bigwedge^{p,q}\g^{\ast}=\bigwedge^{p,q}\langle dz_{1}, e^{-z_{1}}dz_{2}, e^{z_{1}}dz_{3}\rangle\otimes \langle d\bar z_{1}, e^{-\bar z_{1}}d\bar z_{2}, e^{\bar z_{1}}d\bar z_{3}\rangle.
\]
We have 
\begin{multline*}
A^{p,q}
\\=\bigwedge^{p,q}\langle dz_{1}, e^{-z_{1}}dz_{2}, e^{z_{1}}dz_{3}\rangle\bigotimes \langle d\bar z_{1}, e^{-\bar z_{1}}d\bar z_{2}\otimes e^{-2\sqrt{-1}y_{1}}v_{e^{2\sqrt{-1}y_{1}}}, e^{\bar z_{1}}d\bar z_{3}\otimes e^{2\sqrt{-1}y_{1}}v_{e^{-2\sqrt{-1}y_{1}}}\rangle
\end{multline*}
for $z_{1}=x_{1}+\sqrt{-1}y_{1}$.

If $b,d \in \pi\Z$, then we have 
\[H^{p,q}(G/\Gamma)\cong B^{p,q}=\bigwedge^{p,q}\langle dz_{1}, e^{-z_{1}}dz_{2}, e^{z_{1}}dz_{3}\rangle\otimes \langle d\bar z_{1}, e^{- z_{1}}d\bar z_{2}, e^{ z_{1}}d\bar z_{3}\rangle.\]

If $b\not \in \pi\Z$ or $c\not \in\pi\Z$, then we have
\[B^{0,1}=\langle d\bar z_{1}\rangle,\,  B^{0,2}=\langle d\bar z_{2}\wedge d\bar z_{3}\rangle,\, B^{0,3}=\langle d\bar z_{1}\wedge d\bar z_{2}\wedge d\bar z_{3}\rangle
\]
and 
\[H^{p,q}(G/\Gamma)\cong B^{p,q}=\bigwedge^{p}\langle dz_{1}, e^{-z_{1}}dz_{2}, e^{z_{1}}dz_{3}\rangle\otimes B^{0,q}.
\]

 {\bf  Acknowledgements.} 

The author would like to express his gratitude to   Toshitake Kohno for helpful suggestions and stimulating discussions.
He would also like to thank  Takumi Yamada for  helpful comments.
This research is supported by JSPS Research Fellowships for Young Scientists.

\end{document}